\documentclass[12pt]{amsart}
\usepackage{amssymb,latexsym}
\usepackage{enumerate}
\usepackage{float}

\usepackage{mathtools}
\usepackage{amsfonts}
\usepackage{amssymb}
\usepackage{amsmath}
\usepackage{graphicx}
\usepackage{graphicx,color}

\makeatletter
\@namedef{subjclassname@2010}{%
  \textup{2010} Mathematics Subject Classification}
\makeatother

\begin{document}

\title[Stability of cycles by delayed feedback control]{On the stability of cycles by delayed feedback control}

\author{D. Dmitrishin}
\address{Odessa National Polytechnic University, 1 Shevchenko Avenue, Odessa 65044, Ukraine}
\email{dmitrishin@opu.ua}
\author{P. Hagelstein}
\address{Department of Mathematics, Baylor University, Waco, Texas 76798}
\email{paul\!\hspace{.018in}\_\,hagelstein@baylor.edu}
\thanks{This work was partially supported by a grant from the Simons Foundation (\#208831 to Paul Hagelstein).}
\author{A. Khamitova}
\address{Georgia Southern University, Department of Mathematical Sciences,  203 Georgia Avenue,
 Statesboro, Georgia 30460-8093}
\email{anna\!\hspace{.018in}\_\,khamitova@georgiasouthern.edu}
\author{A. Stokolos}
\address{Georgia Southern University, Department of Mathematical Sciences,  203 Georgia Avenue,
 Statesboro, Georgia 30460-8093}
\email{astokolos@georgiasouthern.edu}

\date{}


\dedicatory{Dedicated to Professor Yuriy Kryakin on the occasion
     of his $60$th birthday}

\begin{abstract}

We present a delayed feedback control (DFC) mechanism for stabilizing cycles of one dimensional discrete time systems.   In particular, we consider a delayed feedback control for stabilizing $T$-cycles of a differentiable function $f: \mathbb{R}\rightarrow\mathbb{R}$ of the form
$$x(k+1) = f(x(k)) + u(k)$$
where
$$u(k) = (a_1 - 1)f(x(k)) + a_2 f(x(k-T)) + \cdots + a_N f(x(k-(N-1)T))\;,$$
with $a_1 + \cdots + a_N = 1$.   Following an approach of Morg\"ul, we construct a map $F: \mathbb{R}^{T+1} \rightarrow \mathbb{R}^{T+1}$ whose fixed points correspond to $T$-cycles of $f$.  We then analyze the local stability of the above DFC mechanism by evaluating the stability of the corresponding equilibrum points of $F$.   We associate to each periodic orbit of $f$ an explicit  polynomial whose Schur stability corresponds to the stability of the DFC on that orbit.   An example indicating the efficacy of this method is provided.
\end{abstract}

\maketitle

\newcommand{\ind}{1\hspace{-2.3mm}{1}}

\section{Introduction}
\newtheorem{thm}{Theorem}

The control of chaotic systems is a subject of considerable interest in a number of disciplines, in particular engineering, physics, and mathematics.  In the fundamental paper \cite{ogy90}, Ott, Grebogi, and Yorke observed that chaotic systems frequently contain unstable periodic orbits that may be stabilized by small time-dependent perturbations.   Subsequent papers such as \cite{chen, shinbrot} have explored specific control mechanisms for stabilizing chaotic systems.   One method of control of particular interest is the delayed feedback control (DFC) scheme introduced by Pyragas in \cite{pyragas92}.   The control in the Pyragas scheme is essentially a multiple of the difference between the current and one period delayed states of the system.  Primary advantages of this scheme include the facts that the control term vanishes if the system is already in a periodic orbit and that the control term tends to zero as trajectories approach a given periodic orbit.  The DFC control mechanism  enjoys many applications ranging from the stabilization of the modulation index of lasers to the suppression of pathological brain rhythms  \cite{bdg94, rp04}.

 In spite of its relative simplicity and broad range of application, the stability analysis of the DFC mechanism remains a difficult issue.  A paper of particular interest to us in this regard is one of Morg\"ul.  In \cite{Morgul}, Morg\"ul considers the one-dimensional discrete time system
$$x(k+1) = f(x(k)) + u(k)\;,$$
with $k \in \mathbb{Z}$ being the time index and $f: \mathbb{R} \rightarrow \mathbb{R}$ an appropriately differentiable function.     We suppose that $f$ has a (possibly unstable) $T$-periodic orbit $\Sigma_T = \{x_0^\ast, \ldots, x_{T-1}^\ast\}$, where $f(x_{j\!\mod T}^\ast) = f(x_{j + 1\!\mod T}^\ast)$.  A DFC control used to stabilize the orbit $\Sigma_T$ is$$u(k) = K(x(k) - x(k - T)).$$  As shown by Morg\"ul, we may analyze the local stability of this control by considering the auxiliary function $G: \mathbb{R}^{T+1} \rightarrow \mathbb{R}^{T+1}$ defined by
$G(z_1, \ldots, z_{T+1}) = (z_2, \ldots, z_{T+1}, f(z_{T+1}) + K(z_{T+1} - z_1))$.  We define $F: \mathbb{R}^{T+1} \rightarrow \mathbb{R}^{T+1}$ by $F = G^{T}$ (the composition of $G$ with itself $T$ times.)   Observe that $\Sigma_T':=\{x_0^\ast, \ldots, x^\ast_{T-1}, x_0^\ast\}$ is a fixed point of $F$.  The stability of $\Sigma_T$ under the control mechanism is equivalent to the stability at $\Sigma_T'$ of the system $\hat{x}(k+1) = F(\hat{x}(k))$.  The latter may be analyzed by computing the Jacobian of $F$ at $\Sigma_T'$.   The Jacobian of $F$ has a characteristic polynomial $p(\lambda)$, and the $T$-cycle $\Sigma_T$ is exponentially stable under the control provided that $p(\lambda)$ is Schur stable, i.e. all of its eigenvalues lie inside the unit disc of the complex plane.
 In \cite{Morgul}    Morg\"ul explicitly provided the calculation of the above characteristic polynomial.  Setting $a_1 = f'(x_0^\ast), \ldots, a_T = f'(x_{T-1}^\ast)$, the characteristic polynomial of the Jacobian is of the form
$$p(\lambda) = \lambda^{T+1} + c_T \lambda^T + \cdots + c_1 \lambda + c_0$$
where $c_0 = -(-1)^T K^T, c_T = - \prod_{i=1}^T(a_i + K)$, and
$$c_{T - \ell} = -(-1)^\ell K^\ell \sum_{i_1 = 1}^T \sum_{i_2 = i_1 + 1}^{T} \cdots \sum_{i_\ell = i_{\ell - 1} + 1}^{T} \prod_{i=1\atop i \neq i_1, \ldots, i_\ell}^T(a_i + K)\;.$$
\\

Motivated by this previous work, we wish to consider now a control that takes into account a deeper prehistory of the output values of a function.  In particular, we wish to consider a control of the form

$$u(k) = (a_1 - 1)f(x(k)) + a_2 f(x(k-T)) + \cdots + a_N f(x(k-(N-1)T))\;,$$
where $a_1 + \cdots + a_N = 1$, that takes into account not only the value of $f$ at $x_{k-T}$ but also $x_{k-2T}, x_{k-3T}, \ldots, x_{k-(N-1)T}$.  The reason for considering a control of this form is that, rather than having only one parameter $K$ that may be modified in our attempt to provide stability, we have a set of parameters $a_1, \ldots, a_N$ at our disposal that may be adjusted to provide a more robust control mechanism.
\\

Proceeding analogously to the ideas of Morg\"ul, we associate to the above control a map $G:\mathbb{R}^{T(N-1) + 1} \rightarrow \mathbb{R}^{T(N-1)+1}$ defined by \\

$G(x_1, \ldots, x_{T(N-1) + 1}) =$\\ $(x_2, x_3, \ldots, x_{T(N-1) + 1}, a_1f(x_{T(N-1) + 1}) + a_2f(x_{T(N-2) + 1}) + \cdots + a_Nf(x_1))$.
\\

We  define $F = G^T$, and the stability of the above control on the cycle $\Sigma_T$ may be ascertained by the location of the roots of the characteristic polynomial of the Jacobian of $F$ at $(x_0^\ast, x_1^\ast, \ldots, x_{T(N-1)}^\ast)$ where for convenience we set $x_j^\ast = x^\ast_{{j\mod T}}\;.$   (We will see that the roots of this characteristic polynomial are independent of which point in the $T$-cycle $x_0$ corresponds, a fact that is not \emph{a priori} obvious.)
\\

This approach is all well and good but is of little value if one cannot explicitly compute the characteristic polynomial of the above Jacobian.  The prospect of doing may strike one as rather intimidating, especially seeing the result of Morg\"ul's calculation indicated above when one is considering a control with only one free parameter.   The primary purpose of this note is to show that  the desired characteristic polynomial $p(\lambda)$ may not only be computed but moreover has surprisingly simple form
$$p(\lambda) =\lambda^{(N-1)T + 1} - \mu (q(\lambda))^{T}  \;,$$
where $\mu = f'(x_0^\ast) \cdots f'(x_{T-1}^\ast)$ and
$$q(\lambda) = a_1 \lambda^{N-1} + \cdots + a_{N-1}\lambda + a_N\;.$$


The derivation of the polynomial $p(\lambda)$ was, at least to us, decidedly unobvious, and it seems appropriate for it to be available in the literature.   This is the subject of the next section.   In the subsequent section we will indicate an application of this new control and advantages it has over a control only taking into account a prehistory of length $T$.

\section{Derivation of the polynomial $p(\lambda)$ associated to the control}

We begin by providing the Jacobian associated to the control $$u(k) = (a_1 - 1)f(x(k)) + a_2 f(x(k-T)) + \cdots + a_N f(x(k-(N-1)T))$$ at a point $x^\ast = (x_0^\ast, x_1^\ast, \ldots, x^\ast_{T(N-1)})\in \mathbb{R}^{T(N-1)+1}$, where $\{x^\ast_0, \ldots, x^\ast_{T-1}\}$ forms a $T$-cycle of $f$ and $x^\ast_{j+1} = f(x^*_j)$.   For convenience we assume $f$ is everywhere continuously differentiable and set $\mu_j = f'(x_j^\ast)$ and \newline $\mu = \mu_1 \mu_2 \cdots \mu_T$.  Recalling that
$G(x_1, \ldots, x_{T(N-1) + 1}) =$\\ $(x_2, x_3, \ldots, x_{T(N-1) + 1}, a_1f(x_{T(N-1) + 1}) + a_2f(x_{T(N-2) + 1}) + \cdots + a_Nf(x_1))$, we see by means of induction and application of the chain rule that the Jacobian of $F = G^T$ is given by the $(N-1)T + 1$ by $(N-1)T + 1$ matrix $J$ whose entries are given by

\[\hspace{-.3in}J(i,j) =
\begin{cases}
1 & \textup{if}\; j=i+T\\
a_{N - \lfloor{\frac{j-1}{T}}\rfloor }a_1^{i - (N-2T) - (s+1)} \displaystyle \prod_{k=s}^{i - (N-2)T-1}\mu_k&\textup{if}\; j \equiv s \hspace{-.1in}\mod T ,\;i \geq (N-2)T + 2 , \textup{and}\\
&\;\;\;\;(j-1) \hspace{-.1in}\mod T \leq (i - ((N-2)T + 2)\\
0 & \textup{otherwise}\;.
\end{cases}
\]

The above expression is very nonintuitive, and as a help to the reader we provide an example of such a matrix when $N=3$ and $T=3$.
\begin{figure}[H]
\[ \left( \begin{array}{ccccccc}
0 & 0 & 0 & 1 & 0 & 0 &0 \\
 0 & 0 & 0 & 0 & 1 & 0 &0\\
0 & 0 & 0 & 0 & 0 & 1 &0\\
0 & 0 & 0 & 0 & 0 & 0 &1\\
a_3 \mu_1 & 0 & 0 & a_2 \mu_1 & 0 & 0 &a_1 \mu_1\\
a_1 a_3 \mu_1  \mu_2 & a_3 \mu_2 & 0 & a_1 a_2 \mu_1 \mu_2 & a_2\mu_2 & 0 &a_1^2\mu_1\mu_2\\
a_1^2 a_3 \mu_1\mu_2\mu_3 & a_1 a_3 \mu_2 \mu_3 &  a_3\mu_3 & a_1^2 a_2 \mu_1\mu_2\mu_3 & a_1 a_2\mu_2 \mu_3 & a_2 \mu_3 &a_1^3 \mu_1 \mu_2 \mu_3\\
 \end{array} \right)\]
 \caption{ Jacobian Matrix for $N$=3, $T$=3}
\end{figure}

The following lemma and subsequent theorem constitute the heart of the mathematical content of this paper.
\newtheorem{lem}{Lemma}
\begin{lem}
Let $J$ be the $(N-1)T + 1$ by $(N-1)T + 1$ matrix given above.  The eigenvalues of $J$ are precisely the roots of the polynomial $$p(\lambda) = \lambda^{(N-1)T + 1} - (\mu_1\cdots\mu_T)(a_1 \lambda^{N-1} + \cdots + a_{N-1}\lambda + a_N)^{T}\;.$$

\end{lem}
\begin{proof}

We begin by quickly dispatching the $T=1$ case.   Observe that when $T=1$ the Jacobian matrix takes on the form
\[\left(\begin{array}{cccccc}
0&1&0&0&\cdots&0
\\0&0&1&0&\cdots&0
\\\vdots
\\0&0&0&&&1\\
\\a_N\mu_1&a_{N-1}\mu_1&a_{N-2}\mu_1&\cdots&&a_1\mu_1\\
\end{array}\right)\]

\noindent This is a classical ``companion matrix'' (see, e.g. \cite{Herstein}) and its eigenvalues are the roots of the polynomial
$$p(\lambda) = \lambda^N - \mu_1( a_1\lambda^{N-1} + a_2\lambda^{N-2} + \cdots + a_N)\;,$$ as desired.

In the following we assume that $T \geq 2$.  Let $\lambda$ be an eigenvalue of $J$.  Then $J (x_1, \ldots, x_{(N-1)T + 1})^\textup{T} = \lambda (x_1, \ldots, x_{(N-1)T + 1})^\textup{T}$ for some nonzero vector $(x_1, \ldots, x_{(N-1)T + 1})^\textup{T}$.  (Here to avoid confusion we are using the upright ``\textup{T}'' to denote the transpose of a matrix.) For our convenience, we set
 $$q(\lambda) = a_N + a_{N-1}\lambda + \cdots + a_1 \lambda^{N-1}\;$$
 \noindent and
 $$r(\lambda) = q(\lambda) - a_1\lambda^{N-1}\;.$$
 As $J(e_k) = e_{k+T}$ for $k=1, \ldots, T-k$, we immediately have that $\lambda$ must satisfy the following simultaneous equations for nonzero $(x_1, \ldots, x_T)^\textup{T}$.

\[\begin{cases}
x_1\mu_1 q(\lambda) = \lambda^{N-1}x_2
\\x_1a_1\mu_1\mu_2q(\lambda) + x_2\mu_2r(\lambda)= \lambda^{N-1}x_3
\\\vdots
\\x_1a_{1}^{T-2}\mu_1\cdots\mu_{T-1}q(\lambda) + (x_2a_1^{T-3}\mu_2\cdots\mu_{T-1} + \cdots + x_{T-1}\mu_{T-1})r(\lambda) = \lambda^{N-1}x_T
\\x_1a_1^{T-1}\mu_1\cdots\mu_Tq(\lambda) + (x_2a_1^{T-2}\mu_2\cdots\mu_T + \cdots +x_T\mu_T)r(\lambda) = \lambda^Nx_1\;.
\end{cases}\]
\noindent This holds hold for nonzero $(x_1, \ldots, x_T)^\textup{T}$ if and only if the determinant of the  $T$ by $T$ matrix
\[\hspace{-.5in}\left(\begin{array}{ccccccc}
\mu_1 q(\lambda) & -\lambda^{N-1} & 0 & 0 &  &  &0 \\
 a_1\mu_1\mu_2q(\lambda) & \mu_2r(\lambda) & -\lambda^{N-1} & 0 &  &  &0\\
a_1^2\mu_1\mu_2\mu_3q(\lambda) & a_1\mu_2\mu_3r(\lambda) & \mu_3r(\lambda) & -\lambda^{N-1} &  & \cdots &0\\
\vdots&\vdots&&&&&\vdots\\
a_1^{T-2}\mu_1\cdots\mu_{T-1}q(\lambda)&a_1^{T-3}\mu_2\cdots\mu_{T-1}r(\lambda)&\cdots&&&\mu_{T-1}r(\lambda)&-\lambda^{N-1}
\\
 a_1^{T-1}\mu_1\cdots\mu_Tq(\lambda) - \lambda^N& a_1^{T-2} \mu_2 \mu_3\cdots\mu_Tr(\lambda) &  \cdots & &  &a_1\mu_{T-1}\mu_Tr(\lambda) &\mu_T r(\lambda) \\
 \end{array} \right)\]
 \\
\noindent equals 0.  It suffices to show that this determinant is equal to the polynomial $-p(\lambda) = -\lambda^{(N-1)T +1} + (\mu_1\cdots\mu_T) (q(\lambda))^T\;.$

We proceed by induction on $T$.  Note that
\[\det \left(
\begin{array}{cc}
\mu_1 q(\lambda)&-\lambda^{N-1} \\
a_1 \mu_1\mu_2q(\lambda) - \lambda^N&\mu_2r(\lambda)
 \end{array} \right)\]
 $$= \mu_1\mu_2(q(\lambda)^2 - \lambda^{2N-1}\;.$$
 So the $T=2$ case holds.

 Suppose now that the $T=(k-1)$ case holds, i.e.
 \[\det\left(
 \begin{array}{ccccc}
 \mu_1q(\lambda)&-\lambda^{N-1}&0&\cdots&
 \\a_1\mu_1\mu_2q(\lambda)&\mu_2r(\lambda)&-\lambda^{N-1}&\cdots&\\
 \\\vdots&&&&\vdots\\
 \\&&&&-\lambda^{N-1}
 \\a_{1}^{k-1}\mu_1\cdots\mu_{k-1}q(\lambda) - \lambda^{N}&a_1^{k-2}\mu_2\cdots\mu_{k-1}r(\lambda)&\cdots&&\mu_{k-1}r(\lambda)
 \end{array}\right)\]
$$= \mu_1\mu_2\cdots\mu_{k-1}(q(\lambda))^{k-1} - \lambda^{(N-1)(k-1)+1}\;.$$

Then, taking advantage of the fact that
\[\det\left(
\begin{array}{ccccc}
x_1&&&&0
\\&x_2&&&
\\\ast&&\ddots&&
\\&&&&x_n
\end{array}\right) = x_1\cdots x_n\;,\]
we have that
\[\det\left(
 \begin{array}{ccccc}
 \mu_1q(\lambda)&-\lambda^{N-1}&0&\cdots&
 \\a_1\mu_1\mu_2q(\lambda)&\mu_2r(\lambda)&-\lambda^{N-1}&\cdots&\\
 \\\vdots&&&&\vdots\\
 \\&&&&-\lambda^{N-1}
 \\a_{1}^{k-1}\mu_1\cdots\mu_{k-1}q(\lambda) &a_1^{k-2}\mu_2\cdots\mu_{k-1}r(\lambda)&\cdots&&\mu_{k-1}r(\lambda)
 \end{array}\right)\]
 \\
$$= \mu_1\mu_2\cdots\mu_{k-1}(q(\lambda))^{k-1} - \lambda^{(N-1)(k-1)+1} + \lambda^{N}(\lambda^{N-1})^{k-2}$$
$$=\mu_1\cdots\mu_{k-1}(q(\lambda))^{k-1}\;.$$

So

 \[\det\left(
 \begin{array}{ccccc}
 \mu_1q(\lambda)&-\lambda^{N-1}&0&\cdots&
 \\a_1\mu_1\mu_2q(\lambda)&\mu_2r(\lambda)&-\lambda^{N-1}&\cdots&\\
 \\\vdots&&&&\vdots\\
 \\&&&&-\lambda^{N-1}
 \\a_{1}^{k}\mu_1\cdots\mu_{k}q(\lambda) - \lambda^{N}&a_1^{k-1}\mu_2\cdots\mu_{k}r(\lambda)&\cdots&&\mu_{k}r(\lambda)
 \end{array}\right)\]

 \[=\mu_k r(\lambda) \det\left(
 \begin{array}{ccccc}
 \mu_1q(\lambda)&-\lambda^{N-1}&0&\cdots&
 \\a_1\mu_1\mu_2q(\lambda)&\mu_2r(\lambda)&-\lambda^{N-1}&\cdots&\\
 \\\vdots&&&&\vdots\\
 \\&&&&-\lambda^{N-1}
 \\a_{1}^{k-1}\mu_1\cdots\mu_{k-1}q(\lambda)  &a_1^{k-2}\mu_2\cdots\mu_{k-1}r(\lambda)&\cdots&&\mu_{k-1}r(\lambda)
 \end{array}\right)\]
 \[+ \lambda^{N-1}\det\left(
 \begin{array}{ccccc}
 \mu_1q(\lambda)&-\lambda^{N-1}&0&\cdots&
 \\a_1\mu_1\mu_2q(\lambda)&\mu_2r(\lambda)&-\lambda^{N-1}&\cdots&\\
 \\\vdots&&&&\vdots\\
 \\a_1^{k-2}\mu_1\cdots\mu_{k-2}q(\lambda)&a_1^{k-3}\mu_2\cdots\mu_{k-2}r(\lambda)&\cdots&&-\lambda^{N-1}
 \\a_{1}^{k}\mu_1\cdots\mu_{k}q(\lambda) - \lambda^{N}&a_1^{k-1}\mu_2\cdots\mu_{k}r(\lambda)&\cdots&&a_1\mu_{k-1}\mu_{k}r(\lambda)
 \end{array}\right)\]
$$=\mu_k r(\lambda)\left[\mu_1\cdots\mu_{k-1}(q(\lambda))^{k-1}\right]$$
$$\hspace{1.6in}+ \lambda^{N-1}\left[a_1\mu_1\cdots\mu_k(q(\lambda))^{k-1} - \lambda^{(N-1)(k-1) + 1}\right]\;\;\;\;\textup{(induction)} $$
$$=\mu_1\cdots\mu_k(q(\lambda)^k - \mu_1\cdots\mu_ka_1\lambda^{N-1}(q(\lambda))^{k-1} + \mu_1\cdots\mu_ka_1\lambda^{N-1}(q(\lambda))^{k-1} - \lambda^{(N-1)k + 1}$$
$$=\mu_1 \cdots \mu_k (q(\lambda))^{k} - \lambda^{(N-1)k + 1}\;.$$
Hence by induction we see the result holds.
\end{proof}

\begin{thm}
Let $J$ be the $(N-1)T + 1$ by $(N-1)T + 1$ matrix given above.  The characteristic polynomial of $J$ is the  polynomial $$p(\lambda) = \lambda^{(N-1)T + 1} - (\mu_1\cdots\mu_T)(a_1 \lambda^{N-1} + \cdots + a_{N-1}\lambda + a_N)^{T}\;.$$
\end{thm}

\begin{proof}
The characteristic polynomial of $J$ is of course given by $\det{(\lambda I - J)}$, where $I$ is the $(N-1)T + 1$ by $(N-1)T + 1$ identity matrix.   Observe that this is a polynomial of degree $(N-1)T + 1$ with leading coefficient 1.   Hence, $p(\lambda)$ will be the characteristic polynomimal of $J$, provided all the roots of $p(\lambda)$ are distinct, as by the lemma $p(\lambda)$ is a polynomial of degree $(N-1)T + 1$ with leading coefficient 1 and sharing the same roots as the characteristic polynomial of $J$.

We now deal with the case that $p(\lambda)$ has a root of higher multiplicity.   The idea we employ is that there should exist a minor perturbation of $a_1, \ldots, a_N$ that leads to a new polynomial with distinct roots, and the continuity of the characteristic polynomial of $J$ in terms of $a_1, \ldots, a_N$ should then imply that $p(\lambda)$ is indeed the desired characteristic polynomial of $J$.

Proceeding formally along these lines, we fix $a_1, \ldots, a_N$, and for $z \in \mathbb{C}$ we define the matrix $J_{z,w}$ as the $(N-1)T + 1$ by $(N-1)T + 1$ matrix whose entries are given by
\fontsize{8}{8}
\[\hspace{-.3in}J_{z,w}(i,j) =
\begin{cases}
1 & \textup{if}\; j=i+T\\
a_{z,w,N - \lfloor{\frac{j-1}{T}}\rfloor }a_{z,w,1}^{i - (N-2T) - (s+1)}\displaystyle \prod_{k=s}^{i - (N-2)T-1}\mu_k& j \equiv s \hspace{-.1in}\mod T ,\;i \geq (N-2)T + 2 , \textup{and}\\
&\;\;\;\;(j-1) \hspace{-.1in}\mod T \leq (i - ((N-2)T + 2)\\
0 & \textup{otherwise}\;.
\end{cases}
\]
\fontsize{12}{12}

\noindent where $a_{z,w,j} = za_j$ for $j=1, \ldots, N-1$ and $a_{z,w,N} = a_N + w$.
This is simply the previous Jacobian matrix $J$ but where each $a_i$ has been replaced by $za_i$ for $i=1, \ldots, N-1$ and $a_N$ by $a_N + w$.    Now, it suffices to show that there exists values of $z$ arbitrarily close to 1 and $w$ close to 0 for which $J_{z,w}$ has no repeated eigenvalues.   Equivalently, it suffices to show there exist values of $z$ arbitrarily close to 1 and $w$ close to 0 for which the polynomial

$$p_{z,w}(\lambda) = \lambda^{(N-1)T + 1} - (\mu_1\cdots\mu_T)(za_1 \lambda^{N-1} + \cdots + za_{N-1}\lambda + a_N + w)^{T}\;.$$
has no repeated roots.    To show this, we take advantage of the fact that the polynomial $p_{z,w}(\lambda)$ will have a repeated root if and only if its resultant $R(p_{z,w}, p_{z,w}')$ is zero. (See, e.g., \cite{dummitfoote}.) Recall that the resultant $R(f,g)$ of the polynomials $f(x) = a_n x^n + a_{n-1}x^{n-1} + \cdots + a_1 x + a_0$ and $g(x) = b_m x^m + b_{m-1}x^{m-1} + \cdots + b_1 x + b_0$ is the determinant of the Sylvester matrix
\[\left(\begin{array}{cccccccc}
a_n&a_{n-1}&\cdots&a_0&&&&
\\&a_n&a_{n-1}&\cdots&a_0&&&
\\&&a_n&a_{n-1}&\cdots&a_0&&
\\&&&\ddots&&&&
\\&&&&a_n&a_{n-1}&\cdots&a_0
\\b_m&b_{m-1}&\cdots&b_0&&&
\\&b_m&b_{m-1}&\cdots&b_0&&&
\\&&b_m&b_{m-1}&\cdots&b_0&&
\\&&&\ddots&&&&
\\&&&&\ddots b_m&b_{m-1}&\cdots&b_0
\end{array}\;\right)\;.\]
We then observe that the resultant $R(p_{z,w}, p_{z,w}')$ may be viewed as a holomorphic function in $\mathbb{C}^2$.    Now, if $p_{z,w}(\lambda)$ had repeated roots for all $z$ sufficiently close to 1 and all $w$ close to 0, then $R(p_{z,w} p_{z,w}')$ would be identically zero for all $(z,w) \in \mathbb{C}^2$.   (See, e.g., \cite{Krantz,steinshakarchi4}.)   Setting $z=0$ and $w = 1 - a_N$, we would then have that the polynomial $\lambda^{(N-1)T + 1} - 1$ had repeated roots, a clear contradiction.

\end{proof}

\section{Stabilization of 1-cycles}

We now illustrate the effectiveness of the control
$$u(k) = (a_1 - 1)f(x(k)) + a_2 f(x(k-T)) + \cdots + a_N f(x(k-(N-1)T))\;$$ in the case that $T=1$
 by showing that for any value $\mu := \mu_1 < 1$ the characteristic polynomimal of the associated Jacobian, namely
 $$p(\lambda) = \lambda^{N} - \mu(a_1 \lambda^{N-1} + \cdots + a_{N-1}\lambda + a_N)\;,$$
 has  roots only in the unit disc for appropriately chosen $N$ and $a_1, \ldots, a_N$.  We remark that the control will automatically be unstable for $\mu \geq 1$ for any choice of $N$ and $a_i$ with $a_1 + \cdots + a_N = 1$ as the associated characteristic polynomial $p(\lambda)$ would satisfy $p(1) = 1 - \mu \leq 0$, thus failing the Jury test.
Moreover, if $0 \leq \mu < 1$ we may simply set $a_1 = 0, \ldots a_{N-1} = 0, a_N = 1$ for any $N \geq 1$ to achieve a polynomial all of whose roots lie in the unit disk.  So it suffices to consider  only the case that $\mu < 0$.

  Let $N$ be a positive integer.   We set $a_1, \ldots, a_N$ to all be $1/N$.     A root $\lambda$  of the above polynomial would then need to satisfy the equation
    $$0 = \lambda^{N} - \frac{\mu}{N}(\lambda^{N-1} + \cdots + \lambda + 1)\;.$$

    Now, if $\mu = 0$ the only solutions of the above equation is $\lambda = 0$ which of course lies in the unit disc. Moreover, the complex roots of a polynomial are continuous in the coefficients of the polynomial (see, e.g., \cite{harrismartin}) and accordingly the solutions of the above equation will lie inside the unit disc for all values of $\mu$ down to the greatest (but negative) value of $\gamma$ such that
     $$0 = \lambda^{N} - \frac{\gamma}{N}(\lambda^{N-1} + \cdots + \lambda + 1)\;$$
  has a root $\lambda$ that lies on the boundary of the unit disc.  $\gamma$ is given by the equation
  $$\frac{1}{\gamma}= \frac{1}{N}\inf\{\Re(\lambda^{-N} + \cdots + \lambda^{-1}) : \lambda \in \partial\mathbb{D}\;\textup{and}\;\Im(\lambda^{-N} + \cdots + \lambda^{-1})= 0\}\;.$$
 Expressing values on the boundary of the unit disc as $e^{i\theta}$ for $\theta \in \mathbb{R}$, we have

  $$\frac{1}{\gamma}= \frac{1}{N}\inf\{\Re(e^{-iN\theta} + \cdots + e^{-i\theta}) : \theta \in \mathbb{R}\;\textup{and}\;\Im(e^{-iN\theta} + \cdots + e^{-i\theta})= 0\}\;.$$
If $N$ is odd, we may reexpress $e^{-iN\theta} + \cdots + e^{-i\theta}$ as
\begin{align}
e^{-iN\theta} + \cdots + e^{-i\theta} &= e^{-i\frac{(N+1)}{2}\theta}\left(e^{-i\frac{N-1}{2}\theta} + \cdots + e^{i\frac{N-1}{2}\theta}\right)\notag
\\&=e^{-i\frac{(N+1)}{2}\theta}\left(\frac{\sin \frac{N\theta}{2}}{\sin \frac{\theta}{2}}\right)\;\notag
\end{align}
taking advantage of the closed form for the Dirichlet kernel.  We see that the imaginary part of the above expression is zero  only if $\theta$ is either of the form $\frac{2\pi k}{N}$ for some integer $k$ (for which $\sin(\frac{N\theta}{2}) = 0$) or of the form $\frac{2\pi k}{N+1}$ for some integer $k$ (for which the imaginary part of the exponential term vanishes.)   Exploiting the symmetry of the summands about the origin, we have that  $e^{-iN\theta} + \cdots + e^{-i\theta} = 0$ if $\theta = \frac{2\pi k}{N}$ and also that  $e^{-iN\theta} + \cdots + e^{-i\theta} = -1$ if $\theta = \frac{2\pi k}{N+1}$ .   Hence
$$-\frac{1}{N} = \frac{1}{N}\inf\{\Re(e^{-iN\theta} + \cdots + e^{-i\theta}) : \theta \in \mathbb{R}\;\textup{and}\;\Im(e^{-N\theta} + \cdots + e^{-i\theta})= 0\}\;.$$
Accordingly, we see that by setting $a_j = 1/N$ for $j=1, \ldots, N$, \emph{the above control will be stable for all values of $\mu$ in $(-N, 1)$.}   In particular, regardless of how large the magnitude of $\mu$, we see that by choosing $N$ sufficiently large the above control will be stable about the equilibrium point associated to $\mu$.


\section{Conclusions and Future Directions}

We have seen that that the stability of the  DFC control mechanism
$$x(k+1) = f(x(k)) + u(k),$$
where
$$u(k) = (a_1 - 1)f(x(k)) + a_2 f(x(k-T)) + \cdots + a_N f(x(k-(N-1)T))\;$$
with $a_1 + \cdots + a_N = 1,$
about a $T$-cycle may be ascertained by considering the Schur stability of a polynomial of the form
$$p(\lambda) = \lambda^{(N-1)T + 1} - \mu(a_1 \lambda^{N-1} + \cdots + a_{N-1}\lambda + a_N)^{T}\;.$$   As an application, we showed in the case of $1$-cycles that  this polynomial is Schur stable for appropriately chosen $N$ and $a_1,\ldots, a_N$  for any given value of $\mu < 1$.  Moreover, in the 1-cycle case we have seen that this may be done where $N$ is of the same order of magnitude as $|\mu|$.

    A natural subsequent problem is the following:\\

\noindent    {\bf{Problem:}} For any given integer $T\geq 2$ and $\mu < 1$, do there exist $N$ and associated $a_1, \ldots, a_N$ where $a_1 + \cdots + a_N = 1$ such that all of the roots of the polynomial $p(\lambda)$ lie in the unit disc?  If so, what would be the smallest value of $N$ for which this could be done?
\\

The reader may at this point be speculating whether or not in the $T \geq 2$ case one may simply choose uniform coefficents \mbox{$a_1 = \cdots = a_N = 1/N$} and then try to prove the roots of $p(\lambda)$ all lie in the unit disc for sufficiently large $N$.  We have  analytic evidence that this cannot be done.  However, Dmitrishin and Khamitova proved in \cite{dk2013} that the $T=1$ case may be done with an improved and optimal value of $N$ on the order of magnitude of $|\mu|^{1/2}$, where the associated $a_j$ are given by the formula
$$a_j = 2 \tan\frac{\pi}{2(N+1)}\cdot\left(1 - \frac{j}{N+1}\right)\cdot \sin\frac{\pi j}{N+1}\;\;j=1, \ldots, N\;.$$
Moreover, in the paper \cite{dkks}, the authors indicate that the $T =2$ case may be done with a value of $N$ on the order of magnitude of \mbox{$|\mu_1 \mu_2|$,} and moreover that the order of magnitude in terms of the exponent of $|\mu_1\mu_2|$ is sharp.  A natural conjecture at this point is that $T$-cycles associated to a multiplier $\mu_1\cdots\mu_T$ may be stabilized by the above control for appropriate values of $a_1$, \ldots, $a_N$ where $N$ is on the order of magnitude of $|\mu_1\cdots\mu_T|^{T/2}$.   The associated proofs of the above results involve delicate estimates associated to trigonometric sums and moreover suggest that refined techniques in harmonic analysis may provide an avenue for affirmatively answering this problem, hence establishing a stability result for the above control regardless of the the length of the cycle $T$ and the magnitude of the associated multiplier $\mu$.  This is a subject of ongoing research.

\section{Acknowledgement}

We wish to thank Ron Stanke for his helpful comments and suggestions regarding the content of this paper.

\end{document}